\documentclass[12pt]{amsart}
\usepackage{mathrsfs}
\usepackage{amssymb,amsmath}
\usepackage{setspace}
\usepackage{color}
\usepackage[all]{xy}
\usepackage{graphics}
\usepackage[english]{babel}
\calclayout
\newtheorem{dummy}{anything}[section]
\newtheorem{theorem}[dummy]{Theorem}

\newtheorem{lemma}[dummy]{Lemma}
\newtheorem{proposition}[dummy]{Proposition}
\newtheorem{corollary}[dummy]{Corollary}
\theoremstyle{definition}

 \newtheorem{example}[dummy]{Example}
 
 \newtheorem{remark}[dummy]{Remark}

\newcommand{\z}{\mathbb Z}
\newcommand{\w}{\widetilde}
\newcommand{\rp}{\mathbb R \mathrm P}
\newcommand{\cp}{\mathbb C \mathrm P}
\newcommand{\rpi}{\rp^{\infty}}
\newcommand{\cpi}{\cp^{\infty}}

\newcommand{\hcp}{H\cp^3}

\begin{document}
\title{A note on the $\z_2$-equivariant Montgomery-Yang correspondence}
\author{Yang Su}
\address{Hua Loo-Keng Key Laboratory of Mathematics
\newline \indent
Chinese Academy of Sciences
\newline\indent
Beijing, 100190, China} \email{suyang{@}math.ac.cn}

\date{August 19, 2009}

\begin{abstract}
In this paper, a classification of free involutions on
$3$-dimensional homotopy complex projective spaces is given. By the
$\z_2$-equivariant Montgomery-Yang correspondence, we obtain all
smooth involutions on $S^6$ with fixed-point set an embedded $S^3$.
\end{abstract}

\maketitle

\section{Introduction}\label{sec:one}
In \cite{M-Y}, Montgomery and Yang established a $1:1$
correspondence between the set of isotopy classes of smooth
embeddings $S^3 \hookrightarrow S^6$, $C_3^3$, and the set of
diffeomorphism classes of smooth manifolds homotopy equivalent to
the $3$-dimensional complex projective space $\cp^3$ ( these
manifolds are called homotopy $\cp^3$). It is known that the latter
set is identified with the set of integers by the first Pontrjagin
class of the manifold. Therefore so is the set $C_3^3$.

In a recent paper \cite{Lv-Li}, Bang-he Li and Zhi L\"u established
a $\z_2$-equivariant version of the Montgomery-Yang correspondence.
Namely, they proved that there is a $1:1$ correspondence between the
set of smooth involutions on $S^6$ with fixed-point set an embedded
$S^3$ and the set of smooth free involutions on homotopy $\cp^3$.
This correspondence gives a way of studying involutions on $S^6$
with fixed-point set an embedded $S^3$ by looking at free
involutions on homotopy $\cp^3$. As an application, by combining
this correspondence and a result of Petrie \cite{Petrie}, saying
that there are infinitely many homotopy $\cp^3$'s which admit
 free involutions, the authors constructed
infinitely many counter examples for the Smith conjecture, which
says that only the unknotted $S^3$ in $S^6$ can be the fixed-point
set of an involution on $S^6$.

\smallskip

In this note we study the classification of the orbit spaces of free
involutions on homotopy $\cp^3$. As a consequence, we get the
classification of free involutions on homotopy $\cp^3$, and further
by the $\z_2$-equivariant Montgomery-Yang correspondence, the
classification of involutions on $S^6$ with fixed-point set an
embedded $S^3$.

\smallskip

The manifolds $X^6$ homotopy equivalent to $\cp^3$ are classified up
to diffeomorphism by the first Pontrjagin class $p_1(X)
=(24j+4)x^2$, $j \in \z$, where $x^2$ is the canonical generator of
$H^4(X;\z)$ (c.~f.~\cite{M-Y}, \cite{Wall6}). We denote the manifold
with $p_1 =(24j+4)x^2$ by $\hcp_j$.

\begin{theorem}\label{thm:one}
The manifold $\hcp_j$ admits a (orientation reversing) smooth free
involution if and only if $j$ is even. On every $\hcp_{2k}$, there
are exactly two free involutions up to conjugation.\footnote{The
same result was also obtained independently by Bang-he Li
(unpublished).}
\end{theorem}

\begin{corollary}\label{coro:one}
An embedded $S^3$ in $S^6$ is the fixed-point set of an involution
on $S^6$ if and only if its Montgomery-Yang correspondence is
$\hcp_{2k}$. For each embedding satisfying the condition, there are
exactly two such involutions up to conjugation.
\end{corollary}

Theorem \ref{thm:one} is a consequence of a classification theorem
(Theorem \ref{thm:two}) of the orbit spaces. Theorem \ref{thm:two}
will be shown in Section $3$ by the classical surgery exact sequence
of Browder-Novikov-Sullivan-Wall. In Section $2$ we show some
topological properties of the orbit spaces, which will be needed in
the solution of the classification problem.

\section{Topology of the Orbit Space}

In this section we summarize the topological properties of the orbit
space of a smooth free involution on a homotopy $\cp^3$. Some of the
properties are also given in \cite{Lv-Li}. Here we give shorter
proofs from a different point of view and in a different logical
order.

Let $\tau$ be a smooth free involution on $\hcp$, a homotopy
$\cp^3$. Denote the orbit manifold by $M$.

\begin{example}
The $3$-dimensional complex projective space $\cp^3$ can be viewed
as the sphere bundle of a $3$-dimensional real vector bundle over
$S^4$. The fiberwise antipodal map $\tau_0$ is a free involution on
$\cp^3$ (c.~f.~\cite[A.1]{Petrie}). Denote the orbit space by $M_0$.
\end{example}

As a consequence of the Lefschetz fixed-point theorem and the
multiplicative structure of $H^*(H\cp^3)$, we have

\begin{lemma}\cite[Theorem 1.4]{Lv-Li}
$\tau$ must be orientation reversing.
\end{lemma}

\begin{lemma}\label{lemma:cohom}
The cohomology ring of $M$ with $\z_2$-coefficients is
$$H^*(M;\z_2)=\z_2[t,q]/(t^3=0, q^2=0),$$ where $|t|=1$, $|q|=4$.
\end{lemma}
\begin{proof}
Note the the fundamental group of $M$ is $\z_2$. There is a
fibration $\w{M} \to M \to \rp^{\infty}$, where $M \to \rp^{\infty}$
is the classifying map of the covering. We apply the Leray-Serre
spectral sequence. Since $\w{M}$ is homotopy equivalent to $\cp^3$,
the nontrivial $E_2$-terms are $E_2^{p,q}=H^p(\rp^{\infty};\z_2)$
for $q=0,2,4,6$. Therefore all differentials $d_2$ are trivial, and
henceforth $E_2=E_3$. Now the differential $d_3 \colon E_3^{0,2} \to
E_3^{3,0}$ must be an isomorphism. For otherwise the multiplicative
structure of the spectral sequence implies that the spectral
sequence collapses at the $E_3$-page, which implies that
$H^*(M;\z_2)$ is nontrivial for $* >6$, a contradiction. Then it is
easy to see that $M$ has the claimed cohomology ring.
\end{proof}

\begin{remark}
There is an exact sequence (cf.~\cite{Brown})
$$H_3(\z/2) \to \z \otimes_{\z[\z/2]} \z_- \to H_2(M) \to H_2(\z/2),$$
where $\z_-$ is the nontrivial $\z[\z_2]$-module. By this exact
sequence, $H_2(M)$ is either $\z_2$ or trivial. $H^2(M;\z_2)\cong
\z_2$ implies that $H_2(M) =0$. This was shown in \cite[Lemma
2.1]{Lv-Li} by geometric arguments.
\end{remark}

Now let's consider the Postnikov system of $M$. Since $\pi_1(M)
\cong \z_2$, $\pi_2(M) \cong \z$ and the action of $\pi_1(M)$ on
$\pi_2(M)$ is nontrivial, following \cite{Baues}, there are two
candidates for the second space $P_2(M)$ of the Postnikov system,
which are distinguished by their homology groups in low dimensions.
See \cite[pp.265]{Olbermann} and \cite[Section 2A]{Su}.

Let $\lambda$ be the free involution on $\cpi$, mapping $[z_0, z_1,
z_2, z_3, \cdots]$ to $[-z_1, z_0, -z_3, z_2, \cdots]$. Let $Q =
(\cpi \times S^{\infty})/(\lambda, -1)$, where $-1$ denotes the
antipodal map on $S^{\infty}$, then there is a fibration $\cpi \to Q
\to \rpi$ which corresponds to the nontrivial $k$-invariant. Lemma
\ref{lemma:cohom} implies that $P_2(M)=Q$ since $Q$ has the same
homology as $M$ in low dimensions.

Let $f_2 \colon M \to Q$ be the second Postnikov map, since
$\pi_i(M) \cong \pi_i(\cp^3)=0$ for $3 \le i \le 6$, $f_2$ is
actually a $7$-equivalence and $Q$ is the $6$-th space of the
Postnikov system of $M$. By the formality of constructing the
Postnikov system, all the orbit spaces have the same Postnikov
system. Therefore we have shown
\begin{proposition}\cite[Lemma 3.2]{Lv-Li} \label{prop:hmtp}
The orbit spaces of free involutions on homotopy $\cp^3$ are all
homotopy equivalent.
\end{proposition}

Now let us consider the characteristic classes of $M$.

\begin{lemma}
The total Stiefel-Whitney class of $M$ is $w(M)=1+t+t^2$, where $t
\in H^1(M;\z_2)$ is the generator.
\end{lemma}
\begin{proof}
The involution $\tau$ is orientation reversing, therefore $M$ is
nonorientable and $w_1(M)=t$. The Steenrod square $Sq^2 \colon
H^4(M;\z_2) \to H^6(M;\z_2)$ is trivial, this can be seen by looking
at $M_0$, whose $4$-dimensional cohomology classes are pulled back
from $S^4$. Therefore the second Wu class is $v_2(M)=0$. Thus by the
Wu formula $w(M)=Sqv(M)$ it is seen that the total Stiefel-Whitney
class of $M$ is $w(M)=1+t+t^2$.
\end{proof}

Let $\pi \colon \hcp \to M$ be the projection map, the
$\pi^*p_1(M)=p_1(\hcp)$.

\begin{lemma}
The induced map $\pi^* \colon H^4(M) \to H^4(\hcp)$ is an
isomorphism.
\end{lemma}
\begin{proof}
Apply the Leray-Serre spectral sequence for integral cohomology to
the fibration $\w{M} \to M \to \rpi$, the $E_2$-terms are
$E_2^{p,q}=H^p(\rpi;\underline{H^q(\w{M})})$. It is known that
$H^3(M)=H^3(Q)=0$ and $H^5(M)=H^5(Q)=0$ (for $H^*(Q)$, see
\cite[pp.265]{Olbermann}), therefore
$E_{\infty}^{0,4}=E_2^{0,4}=H^4(\w{M})$ is the only nonzero term in
the line $p+q=4$. This shows that the edge homomorphism is an
isomorphism, which is just $\pi^*$.
\end{proof}

Therefore the first Pontrjagin class of $M$ is $p_1(M)=(24j+4)u$ ($j
\in \z$), where $u=\pi^*(x^2)$ is the canonical generator of
$H^4(M)$.

\section{Classification of the orbit spaces}
By Proposition \ref{prop:hmtp}, every orbit space $M$ is homotopy
equivalent to $M_0$. Thus the set of conjugation classes of free
involutions on homotopy $\cp^3$ is in $1:1$ correspondence to the
set of diffeomorphism classes of smooth manifolds homotopy
equivalent to $M_0$. Denote the latter by  $\mathcal M (M_0)$. Let
$\mathscr{S}(M_0)$ be the smooth structure set of $M_0$, $Aut(M_0)$
be the set of homotopy classes of self-equivalences of $M_0$. There
is an action of $Aut(M_0)$ on $\mathscr S(M_0)$ with orbit set
$\mathcal M(M_0)$. (Since the Whitehead group of $\z_2$ is trivial,
we omit the decoration $s$ all over.)

The surgery exact sequence for $M_0$ is
$$L_7(\z_2^-) \to \mathscr S(M_0) \to [M_0, G/O] \to L_6(\z_2^-).$$
By \cite[Theorem 13A.1]{Wall}, $L_7(\z_2^-)=0$, $L_6(\z_2^-)
\stackrel{c}{\cong} \z_2$, where $c$ is the Kervaire-Arf invariant.
Since $\dim M_0=6$ and $PL/O$ is $6$-connected, we have an
isomorphism $[M_0, G/O] \cong [M, G/PL]$. For a given surgery
classifying map $g \colon M_0 \to G/PL$, the Kervaire-Arf invariant
is given by the Sullivan formula (\cite{Sullivan}, \cite[Theorem
13B.5]{Wall})
\begin{eqnarray*}
c(M_0, g) & = & \langle w(M_0) \cdot g^*\kappa, [M_0] \rangle \\
          & = & \langle (1+t+t^2) \cdot g^*(1+Sq^2+Sq^2Sq^2)(k_2+k_6),
          [M_0] \rangle \\
          & = & \langle g^*k_6, [M_0] \rangle .
\end{eqnarray*}

Now since $M_0$ has only $2$-torsion, and modulo the groups of odd
order we have
$$G/PL \simeq Y \times \prod_{i \ge 2}(K(\z_2, 4i-2) \times
K(\z,4i)),$$ where $Y=K(\z_2,2) \times_{\delta Sq^2} K(\z,4)$, we
have $[M_0, G/PL]=[M_0,Y] \times [M_0, K(\z_2,6)]$. $k_6$ is the
fundamental class of $K(\z_2,6)$. Therefore the surgery exact
sequence implies

\begin{lemma}\label{lemma:str}
$\mathscr S (M_0) \cong  [M_0, Y]$.
\end{lemma}

The projection $\pi \colon \cp^3 \to M_0$ induces a homomorphism
$\pi^* \colon [M_0, Y] \to [\cp^3, Y]$, and $[\cp^3,Y]$ is
isomorphic to $\z$ through the splitting invariant $s_4$
(\cite[Lemma 14C.1]{Wall}). Let $\Phi=s_4 \circ \pi^*$ be the
composition.

\begin{lemma}\label{lemma:exact}
There is a short exact sequence $\z_2 \to [M_0,Y]
\stackrel{\Phi}{\rightarrow} 2\z$.
\end{lemma}
\begin{proof}
We have $[\cp^3, Y]=[\cp^2, Y]$, and according to Sullivan
\cite{Sullivan}, the exact sequence
$$L_4(1) \stackrel{\cdot 2}{\rightarrow} [\cp^2, Y] \to [\cp^1,Y]$$
is non-splitting. Let $p \colon Y \to K(\z_2,2)$ be the projection
map, then for any $f \in [\cp^3, Y]$, $s_4(f) \in 2\z$ if and only
if $p \circ f \colon \cp^3 \to K(\z_2,2)$ is null-homotopic. Now by
Lemma \ref{lemma:cohom}, the homomorphism $H^2(M_0;\z_2) \to
H^2(\cp^3;\z_2)$ is trivial. Therefore for any $g \in [M_0, Y]$, the
composition $p \circ g \circ \pi$ is null-homotopic, thus
$\mathrm{Im} \Phi \subset 2\z$. On the other hand, since $\pi^*
\colon H^4(M_0;\z) \to H^4(\cp^3)$ is an isomorphism, any map $f
\colon \cp^3 \to K(\z,4)$ factors through some $g' \colon M_0 \to
K(\z,4)$. Let $i \colon K(\z,4) \to Y$ be the fiber inclusion, since
$s_4(i\circ f)$ takes any value in $2\z$, so does $\Phi(i \circ
g')$.

Let $h \colon M_0 \to K(\z_2,2)$ be a map corresponding to the
nontrivial cohomology class in $H^2(M_0;\z_2)$. By obstruction
theory, there is a lifting $g \colon M_0 \to Y$. By the previous
argument, there is also a map $g' \colon M_0 \to Y$ such that
$\Phi(g)=\Phi(g')$, but $ p \circ g' \colon M_0 \to K(\z_2,2)$ is
null-homotopic. Therefore the kernel of $\Phi$ consists of two
elements.
\end{proof}

\begin{remark}
In \cite{Petrie} Petrie showed that every homotopy $\cp^3$ admits
free involution. It was pointed out by Dovermann, Masuda and Schultz
\cite[pp.~4]{DMS} that since the class $G$ is in fact twice the
generator of $H^4(S^4)$, Petrie's computation actually shows that
every $\hcp_{2k}$ admits free involution, which is consistent with
Lemma \ref{lemma:exact}.
\end{remark}

The set of diffeomorphism classes of manifolds homotopy equivalent
to $M_0$, $\mathcal M(M_0)$, is the orbit set $\mathscr
S(M_0)/Aut(M_0)$. In general, the determination of the action of
$Aut(M_0)$ on the structure set is very difficult. But in our case,
the situation is quite simple, since

\begin{lemma}\label{lemma:action}
The group of self-equivalences $Aut(M_0)$ is the trivial group.
\end{lemma}
\begin{proof}
A special CW-complex structure of $M_0$ was given in
\cite[pp.~885]{Lv-Li}: $M_0$ is a $\rp^2$-bundle over $S^4$,
therefore it is the union of two copies of $\rp^2 \times D^4$, glued
along boundaries. Choose a CW-complex structure of $\rp^2$, we have
a product CW-structure on one copy of $ \rp^2 \times D^4$, and by
shrinking the other copy of $\rp^2 \times D^4$ to the core $\rp^2$,
we get a CW-complex structure on $M_0$, whose $2$-skeleton is
$\rp^2$.

Let $\varphi \in Aut(M_0)$ be a self homotopy equivalence of $M_0$.
By cellular approximation, we may assume that $\varphi$ maps $\rp^2$
to $\rp^2$. It is easy to see that $\varphi|_{\rp^2}$ is homotopic
$\mathrm{id}_{\rp^2}$. Therefore, by homotopy extension, we may
further assume that $\varphi|_{\rp^2}=\mathrm{id}_{\rp^2}$. The
obstruction to construct a homotopy between $\varphi$ and
$\mathrm{id}_{M_0}$, which is the identity on $\rp^2$, is in
$H^i(M,\rp^2;\pi_i(M_0))$. Since $\pi_i(M_0)=0$ for $3 \le i \le 6$
and $H^1(M_0,\rp^2;\z_2)=H^2(M_0,\rp^2;\z)=0$, all the obstruction
groups are zero. Therefore $\varphi \simeq \mathrm{id}_{M_0}$.
\end{proof}

Combine Lemma \ref{lemma:str}, Lemma \ref{lemma:exact} and Lemma
\ref{lemma:action}, we have a classification of manifolds homotopy
equivalent to $M_0$.

\begin{theorem}\label{thm:two}
Let $M^6$ be a smooth manifold homotopy equivalent to $M_0$, then
$p_1(M)=(48j+4)u$, where $u\in H^4(M;\z)$ is the canonical
generator; for each $j \in \z$, up to diffeomorphism, there are two
such manifolds with the same $p_1=48j+4$.
\end{theorem}

Theorem \ref{thm:one} and Corollary \ref{coro:one} are direct
consequences of this theorem.

\end{document}